\documentclass[12pt, reqno]{amsart}
\usepackage{prettyref}
\usepackage{mathrsfs}
\usepackage{amsthm}
\usepackage{amssymb, graphicx, stmaryrd, color, fullpage}

\makeatletter
\numberwithin{equation}{section}
\numberwithin{figure}{section}
\theoremstyle{plain}
\newtheorem{thm}{\protect\theoremname}[section]
\theoremstyle{plain}
\newtheorem{lem}[thm]{\protect\lemmaname}
\theoremstyle{plain}
\newtheorem{cor}[thm]{\protect\corollaryname}
\theoremstyle{plain}

\usepackage{bbm,stmaryrd}

\usepackage{hyperref}

\newcommand{\cC}{\mathcal{C}}

\newcommand{\cL}{\mathcal{L}}

\newcommand{\R}{\mathbb{R}}

\newcommand{\bT}{\mathbb{T}}

\newcommand{\Z}{\mathbb{Z}}

\newcommand{\E}{\mathbb{E}}

\newcommand{\abs}[1]{\lvert#1\rvert}
\newcommand{\norm}[1]{\lVert#1\rVert}

\newcommand{\tmrsup}[1]{\textsuperscript{#1}}

\newcommand{\tthree}[1]{#1\tmrsup{\resizebox{.9em}{!}{\includegraphics{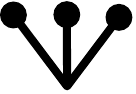}}}}
\newcommand{\ttwo}[1]{#1\,\tmrsup{\!\resizebox{.9em}{!}{\includegraphics{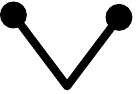}}}}

\newrefformat{cor}{Corollary \ref{#1}}

\newcommand{\para}{\varolessthan}
\newcommand{\arap}{\varogreaterthan}
\newcommand{\reso}{\varodot}

\makeatother

\providecommand{\corollaryname}{Corollary}
\providecommand{\lemmaname}{Lemma}
\providecommand{\theoremname}{Theorem}
\providecommand{\remarkname}{Remark}

\begin{document}
\title{A simple construction of the dynamical $\Phi^4_3$ model}
\author{Aukosh Jagannath}
\address[Aukosh Jagannath]{Department of Statistics and Actuarial Sciences, University of Waterloo}
\email{a.jagannath@uwaterloo.ca}
\author{Nicolas Perkowski}
\address[Nicolas Perkowsi]{Institut f\"ur Mathematik, Freie Universit\"at Berlin}
\email{perkowski@math.fu-berlin.de}

\begin{abstract}
	The $\Phi^4_3$ equation is a singular stochastic PDE with important applications in mathematical physics. Its solution usually requires advanced mathematical theories like regularity structures or paracontrolled distributions, and even local well-posedness is highly nontrivial. Here we propose a multiplicative transformation to reduce the periodic $\Phi^4_3$ equation to a well-posed random PDE. This leads to a simple and elementary proof of global well-posedness, which only relies on Schauder estimates, the maximum principle, and basic estimates for paraproducts, and in particular does not need regularity structures or paracontrolled distributions.
\end{abstract}

\maketitle

\section{Introduction}
In this note we prove the global-in-time existence and uniqueness
of solutions to the $\Phi_{3}^{4}$ equation, which is formally written as
\begin{equation}\label{eq:phi-4-3-formal}
\left(\partial_{t}-\Delta + 1\right)\phi=-\phi^{3}+\xi, \qquad (t,x)\in\R_{+}\times\bT^{3},
\end{equation}
where $\xi$ is a space-time white noise on $\R\times\bT^{3}$ and $\bT^3 = (\R/\Z)^3$ is the three-dimensional torus.
The $\Phi_3^4$ equation is one of the most prominent singular SPDEs and it has at least three important applications: Its invariant measure is given by the $\Phi^4_3$ measure from quantum field theory, and as predicted by the stochastic quantization approach of \cite{Parisi1981} we can use the SPDE to give a dynamic construction of this (highly nontrivial) measure \cite{Gubinelli2021}. The $\Phi^4_3$ equation is also a universal model for $3d$ interface coexistence models near bifurcations \cite{Hairer2018Large, Furlan2017}, and it is expected to describe the dynamics of $3d$ ferromagnets close to their critical temperature \cite{Mourrat2017}.
 
 The rigorous interpretation and local well-posedness of the $\Phi^4_3$ equation had  been an open problem for a long time, because the solution $\phi(t)$ is, for each $t>0$, a distribution in the space variable and therefore the interpretation of the term $-\phi^3$ is unclear. With the development of new mathematical theories such as regularity structures~\cite{Hairer2014} or paracontrolled distributions~\cite{Gubinelli2015Paracontrolled, Catellier2018}, for example, the rigorous construction of $-\phi^3$ became possible and this lead to a local existence-uniqueness result.  Because of the nonlinearity of the equation, global existence is more subtle and it was first shown by Mourrat and Weber~\cite{Mourrat2017Dynamic}, who even proved a ``coming down from infinity'' estimate, i.e., a bound for $\phi(t)$ at $t>0$ which is independent of the initial condition. In recent years simpler proofs have been obtained, which even allow to solve the equation on $\R_{+}\times\R^{3}$ instead of $\R_{+}\times\bT^{3}$; see~\cite{Gubinelli2019Global, Moinat2020}. 

All of these results (even local existence) require sophisticated mathematical tools. Here we suggest a sort of inverse Cole-Hopf transform, previously employed in the context of linear equations for example in~\cite{Hairer2015Simple, Gubinelli2017KPZ}, which leads to a short and relatively elementary proof of global well-posedness of the periodic $\Phi^4_3$ equation. In particular, we do not need paracontrolled distributions or regularity structures. 

In the next section we introduce the transform and we prove the local well-posedness of the transformed equation. Section~\ref{sec:global} contains the proof of global well-posedness, which is similar to the arguments by Gubinelli and Hofmanov\'a for the $\Phi^4_2$ equation~\cite{Gubinelli2019Global}; roughly speaking, the transform reduces the singularity of the equation and brings $\Phi^4_3$ to the same level of difficulty as the much easier $\Phi^4_2$ equation. In Section~\ref{sec:extensions} we discuss possible extensions of our results.

\paragraph{\textit{Notation:}} In the following, for two functions we say that
$f\lesssim_{a}g$ if there is a constant $C(a)>0$ depending only
on $a$ such that $f\leq C(a)g$. We say that $f\simeq g$ if $f\lesssim g$
and $g\lesssim f$. 

\section{Local Solutions}\label{sec:local}
We begin by proving the local well-posedness of the $\Phi_3^4$ equation. 
Recall that the $\Phi_3^4$ equation is to be understood in a renormalized sense.
For this reasons, in place of \eqref{eq:phi-4-3-formal}, the equation is sometimes written formally as
 $(\partial_t-\Delta+1)\phi = -\phi^3 + \infty\cdot\phi+\xi$,
where the $\infty$ denotes a divergent  counter-term.

More precisely, let $\cL=\partial_{t}-\Delta+1$ and let $Z$ be stationary such that
\[
	\cL Z = \xi,
\]
for a two-sided space-time white noise $\xi$ on $\R_+ \times \bT^3$ and let $\cC^\alpha$ denote
the H\"older-Besov space of regularity $\alpha$. (See \eqref{eq:besov-def} below for a precise definition of this space.)
Stated rigorously, our goal is to understand the sequence of solutions to
\[
\begin{cases}
(\partial_t-\Delta+1)\phi_\delta = -\phi_\delta^3 + (3a_\delta-3b_\delta)\phi_\delta +\xi_\delta, & (t,x)\in \R_+\times \bT^3,\\
\phi_\delta(0,x)= Z_\delta(0)+\phi_0^\sharp,
\end{cases}
\]
where $\phi_0^\sharp\in \cC^{3/2-\epsilon}$ for some sufficiently small $\epsilon>0$, $\xi_\delta$ is a suitable mollification of $\xi$,
and $a_\delta,b_\delta$ are (deterministic) sequences of scalars which diverge for $\delta\to0$; also, $\cL$ is stationary with $\cL Z_\delta = \xi_\delta$.
Our goal is to understand the limit $\phi=\lim_{\delta\to0} \phi_\delta$, which we call the renormalization limit.

Since the renormalization procedure is present at every step of our analysis, we will suppress the dependencies
on $\delta$ for the sake of readability. For example, we will simply write that we seek the solution to 
the equation
\begin{equation}\label{eq:Phi-4-3}
\begin{cases}
\left(\partial_{t}-\Delta + 1\right)\phi=-\phi^{3} + (3a-3b)\phi +\xi, & (t,x)\in\R_{+}\times\bT^{3},\\
\phi(0,x)= Z(0) + \phi^\sharp_0.
\end{cases}
\end{equation}
Here and in the following, we use the convention that whenever we write an equation or inequality, 
the expression is to be understood as holding for all $\delta$, with any unquantified constants appearing 
in bounds being independent of $\delta$.

In order to develop solutions to this equation, we need to understand certain terms built from $Z$
which we will refer to as the stochastic tree terms. To define these terms, we recall the following.
Let $\mathscr{S'}$ denote the space of tempered distributions on
$\mathbb{T}^{3}$, and let $\Delta_{j}$ denote the usual Littlewood-Payley operators \cite{Bahouri2011}. Consider the sum $\Delta_{\leq j}=\sum_{i\leq j}\Delta_{i}$,
and define $\Delta_{\geq j},\Delta_{>j},$ and $\Delta_{<j}$ similarly.
For $u,v\in\mathscr{S}'$, consider the \emph{paraproduct}
\[
u\para v=\sum_{j\geq-1}\Delta_{\leq j-2}u\Delta_{j}v
\]
and the\emph{ resonant product }
\[
u\reso v=\sum_{i,j:\abs{i-j}\leq1}\Delta_{i}u\Delta_{j}v.
\]
We denote their sum by $u\preccurlyeq v=u\para v+u\reso v.$

We consider the following nonlinear terms built from $Z$:  
\begin{gather*}
	\llbracket Z^2  \rrbracket := Z^2 - a, \quad
	\cL \tthree{Z} = \llbracket Z^3 \rrbracket := Z^3 - 3 a Z,\qquad\cL  \ttwo{Z} = \llbracket Z^2 \rrbracket, \\
	\tthree{Z}\reso Z,\qquad \ttwo{Z} \reso \llbracket Z^2 \rrbracket -\frac{b}{3}, \qquad |\nabla \ttwo{Z}|^2 - \frac{b}{3}, \qquad \tthree{Z} \reso \llbracket Z^2 \rrbracket -  bZ,
\end{gather*}
where $\ttwo{Z}$ and $\tthree{Z}$ have initial condition $0$: $\ttwo{Z}(0) = \tthree{Z}(0) = 0$.
We remind the reader that these terms all implicitly depend on $\delta$, e.g., the terms $Z,a,b$ should be understood as $Z_\delta,a_\delta,b_\delta$. (For fixed $\delta>0$  the terms $a$ and $b$  in the above expressions are the same as those in \eqref{eq:Phi-4-3}.)

Using Gaussian analysis, it can be shown that each of these terms converges as $\delta\to 0$ to a limit that does not depend on the specific mollification used for $\xi_\delta$.
As the existence of the limits is implicit in any
solution theory to the $\Phi_3^4$ equation, we do not prove this here. Instead, we refer the reader to  \cite{Catellier2018} or, for a pedagogical exposition,  \cite{Mourrat2017Construction}.
 While the term $|\nabla \ttwo{Z}|^2 - \frac{b}{3}$ does not appear in the references, it is (basically) equivalent to $\ttwo{Z} \reso \llbracket Z^2 \rrbracket -\frac{b}{3}$; for the convenience of the reader, we discuss this in the appendix. The regularities of these terms are given in Table~\prettyref{t:reg}, which should be read as $Z_\delta \to Z \in C_T \cC^{-\frac12-\epsilon}$, $\llbracket Z^2_\delta \rrbracket \to \llbracket Z^2 \rrbracket \in C_T \cC^{-1-\epsilon}$, etc., and we can take any $\epsilon > 0$.
 
  We work with the following function spaces:
 \renewcommand{\arraystretch}{1.7}
  \begin{table}
      \begin{center}
  \begin{tabular}{| r | c | c | c | c|c|c|c|c|}
  \hline
  $\tau$ & $Z$ & $ \llbracket Z^2 \rrbracket$ & 
  $\tthree{Z}$ & $\ttwo{Z}$ & $\tthree{Z}\reso Z$  & $\ttwo{Z}\reso \llbracket Z^2 \rrbracket - \frac{b}{3}$ & $|\nabla \ttwo{Z} |^2  - \frac{b}{3}$ & $ \tthree{Z} \reso \llbracket Z^2 \rrbracket - bZ $\\ \hline 
  $\alpha_{\tau}$ & $-\frac{1}{2}-\epsilon$ & $-1-\epsilon$ & 
  $\frac{1}{2}-\epsilon$ & $1-\epsilon$ & $-\epsilon$ & $-\epsilon$ & $-\epsilon$ & $-\frac{1}{2}-\epsilon$ \\
  \hline
  \end{tabular}
  \vspace{10pt}
  \caption{Regularity of stochastic objects.}
  \label{t:reg}
  \end{center}
  \end{table}
\begin{equation}\label{eq:besov-def}
\norm{u}_{\cC^{\alpha}}=\norm{\left(2^{j\alpha}\norm{\Delta_{j}u}_{L^{\infty}}\right)_{j\geq-1}}_{\ell^{\infty}},
\end{equation}
for $\alpha \in \R$ and $C_T \cC^\alpha := C([0,T],\cC^\alpha)$ with $\norm{u}_{C_T\cC^\alpha} := \max_{t \in [0,T]} \norm{ u(t)}_{\cC^\alpha}$.

The usual approach to solving the $\Phi^4_3$ equation starts by considering $v=\phi-Z+\tthree{Z}$. Then we obtain a new equation for $v$, which is better behaved than the equation for $\phi$ because by subtracting $Z - \tthree{Z}$ we removed the most singular terms. However, the equation for $v$ is still ill-posed because it involves the singular product $-3v \llbracket Z^2 \rrbracket$. To deal with this product, previous works used regularity structures or paracontrolled distributions.

Here we propose an alternative approach, by observing that the singular product $-3v \llbracket Z^2 \rrbracket$ can be removed with a multiplicative transform: We consider
\[
v=e^{3\ttwo{Z}}\left(\phi-Z+\tthree{Z} \right).
\]
Then, formally, $\phi$ solves~\eqref{eq:Phi-4-3} if and only if $v$ solves 
\begin{align}
\cL v & =\left(-3\ttwo{Z}+ 9\left(\abs{\nabla\ttwo{Z}}^{2} - \frac{b}{3}\right)\right)v -6(\nabla \ttwo{Z} \cdot\nabla v - b e^{3\ttwo{Z}}\tthree{Z})\label{eq:v-def}\\
 & \quad  + e^{3\ttwo{Z}}\left[3 \tthree{Z} \llbracket Z^2 \rrbracket - 3b\left(Z+\tthree{Z}\right)   - 3Z\left(e^{-3\ttwo{Z}}v-\tthree{Z}\right)^{2} - \left(e^{-3\ttwo{Z}}v-\tthree{Z}\right)^{3} \right],\nonumber 
\end{align}
with initial data $v(0)=\phi(0) - Z(0) = \phi^\sharp_0$. This equation is still ill-posed, because of the singular term $3e^{3\ttwo{Z}} \tthree{Z} \llbracket Z^2 \rrbracket$ on the right hand side. To remove it, we consider
\[
	\cL Y = 3e^{3\ttwo{Z}}( \tthree{Z} \llbracket Z^2 \rrbracket - b(Z + \tthree{Z})),\quad Y(0) = 0.
\]
We show in the appendix that $Y \in \bigcap_{\epsilon>0} C_T\cC^{1-\epsilon}$ can be constructed as a continuous function of the stochastic tree terms described above. Let then $u = v - Y$, which solves
\begin{align}
\cL u & =\left(-3\ttwo{Z}+ 9\left(\abs{\nabla\ttwo{Z}}^{2} - \frac{b}{3}\right)\right)(u+Y) -6\nabla \ttwo{Z} \cdot\nabla u - 6(\nabla \ttwo{Z} \cdot \nabla Y - b e^{3\ttwo{Z}} \tthree{Z})\label{eq:u-def}\\
 & \quad  - e^{3\ttwo{Z}}\left[  3Z\left(e^{-3\ttwo{Z}}(u+Y)-\tthree{Z}\right)^{2} + \left(e^{-3\ttwo{Z}}(u+Y)-\tthree{Z}\right)^{3} \right],\nonumber 
\end{align}
with initial condition $u(0) = v(0) = \phi^\sharp_0$. We show in the appendix that all the explicit (renormalized) nonlinear functions of the stochastic tree terms that appear on the right hand side can be constructed as continuous functions of the given trees. Expanding the nonlinearities, we can then rewrite the equation for $u$ as
\begin{equation}
\cL u=-6\nabla \ttwo{Z} \cdot\nabla u - e^{-6\ttwo{Z}}u^{3}+Z_{2}u^{2}+Z_{1}u+Z_{0},\label{eq:u-def-rewrite}
\end{equation}
for some $Z_{0},Z_{1}$, and $Z_{2}$ that are in $C_{T}\cC^{-1/2-\epsilon}$
for every $T,\epsilon>0$. 

After these preparations, the following local existence/uniqueness result is easy. For the remainder of the paper, we fix $\epsilon > 0$ which is small enough so that all of the following arguments work ($\epsilon < 1/6$ is sufficient).

\begin{thm}\label{thm:local-solutions}
Assume that $u(0) = \phi^\sharp_0\in \cC^{3/2-\epsilon}$. There exists $T^\ast > 0$ such that for all $T< T^\ast$ the equation~\prettyref{eq:u-def} has a unique solution $u \in C_T \cC^{3/2-\epsilon}$, and this solution depends continuously on $\Z = (\ttwo Z, Z_0, Z_1, Z_2) \in \bigcap_{\epsilon'>0} \left(C_T\cC^{-\epsilon'} \times (C_T \cC^{-1/2-\epsilon'})^3\right)$ and $\phi_0^\sharp \in \cC^{3/2-\epsilon}$. Thus, in particular, its limit for $\delta\to 0$ exists and  is independent of the specific mollification used for $\xi_\delta$. Moreover, if $T^\ast < \infty$, then $\lim_{t \uparrow T^\ast} \norm{u(t)}_{\cC^{3/2-\epsilon}} = \infty$.
\end{thm}

Before we get to the proof, recall that for $\alpha+\beta>0$, the map
$(u,v)\mapsto uv$ extends to a bounded linear form on $\cC^{\alpha}\times\cC^{\beta}\to\cC^{\alpha\wedge\beta}$, see~\cite{Bahouri2011}.
In particular, 
\begin{equation}
\norm{uv}_{\cC^{\alpha\wedge\beta}}\leq\norm{u}_{\cC^{\alpha}}\cdot\norm{v}_{\cC^{\beta}}.\label{eq:product-norm}
\end{equation}
Moreover, $\nabla$ is a bounded linear operator from $\cC^\alpha$ to $\cC^{\alpha-1}$ for any $\alpha \in \R$~\cite{Bahouri2011}.

Also, let $(P_t)_{t\ge 0}$ be the semigroup generated by $\cL$. Recall the following estimate for $(P_t)_{t\ge 0}$~\cite{Gubinelli2015Paracontrolled, Gubinelli2019Global}: for
$\gamma\geq0$ and $\alpha\in\R$
\begin{equation}\label{eq:semigroup-estimate}
	\norm{P_{t}w}_{\cC^{\alpha+\gamma}}\lesssim e^{-t} t^{-\gamma/2} \norm{w}_{\cC^{\alpha}} \le t^{-\gamma/2} \norm{w}_{\cC^{\alpha}}, \qquad t \ge 0.
\end{equation}

\begin{proof}
We work with the formulation~\prettyref{eq:u-def-rewrite}. The regularity of $\nabla \ttwo{Z}$ is $\cC^{-\epsilon}$ and therefore $-6\nabla \ttwo{Z} \cdot\nabla u$ is well-defined if $u\in \cC^{\alpha}$ for $\alpha>1+\epsilon$. Similarly, $Z_1, Z_2$ have regularity $\cC^{-1/2-\epsilon/2}$, and therefore the remaining products are well-defined if $u\in \cC^{\alpha}$ for $\alpha>1/2+\epsilon/2$. In that case, the right hand side of~\prettyref{eq:u-def-rewrite} will
be in $\cC^{-1/2-\epsilon/2}$. By Schauder estimates, we then expect
that we can take $\alpha=3/2-\epsilon$, and thus the PDE is locally well-posed.
We make this precise by a standard Picard iteration argument as follows. 

Consider the map 
\[
F(u)(t) = P_t \phi^\sharp_0 + \int_{0}^{t}P_{t-s}\left(-6\nabla \ttwo{Z} \cdot\nabla u - e^{-6\ttwo{Z}}u^{3}+Z_{2}u^{2}+Z_{1}u+Z_{0}\right)(s)ds.
\]
Applying the semigroup estimate~\prettyref{eq:semigroup-estimate} with $\alpha=-1/2-\epsilon/2$ and $\gamma=2-\epsilon/2$, we have for
any $u,v \in C_{T}\cC^{3/2-\epsilon}$, 
\begin{align*}
\norm{&F(v)-F(u)}_{C_{T}C^{3/2-\epsilon}} \\
& \lesssim T^{\epsilon/4}\cdot\norm{-6\nabla \ttwo{Z} \cdot\nabla (v-u) - e^{-6\ttwo{Z}}(v^{3}-u^3)+Z_{2}(v^2-u^{2})+Z_{1}(v-u)}_{C_{T}\cC^{-\frac{1}{2}-\frac{\epsilon}{2}}}\\
 & \lesssim_{\Z} T^{\epsilon/4}(1+\norm{v}_{C_{T}\cC^{3/2-\epsilon}}^2+\norm{u}_{C_{T}\cC^{3/2-\epsilon}}^2) \norm{v-u}_{C_{T}\cC^{3/2-\epsilon}},
\end{align*}
where in the second line we used the product estimate \eqref{eq:product-norm}. Let $M = 2\|\phi^\sharp_0\|_{\cC^{3/2-\epsilon}}$. Then for $T>0$
sufficiently small (depending only on $M$ and $\Z$), the map $F$ is a contraction on the ball of radius $2M$ in
$C_{T}\cC^{3/2-\epsilon}$. The fixed point is then our desired solution on $[0,T]$. By iterating this construction on $[T,T+T']$, for $T'>0$ possibly smaller than $T$, etc., we can extend the solution up to a possibly finite explosion time $T^\ast$.

The continuous dependence on the data is shown with similar arguments: For $\Z, \tilde \Z$ and $\phi_0^\sharp, \tilde \phi_0^\sharp$, we consider the fixed points $u, \tilde u$ of the maps $F_{\Z, \phi_0^\sharp}, F_{\tilde\Z, \tilde \phi^0_\sharp}$ and write $u-\tilde u = F_{\Z, \phi_0^\sharp}(u)- F_{\tilde\Z, \tilde \phi^0_\sharp}(\tilde u)$ to bound $u-\tilde u$ in terms of $\Z - \tilde \Z$ and $\phi_0^\sharp - \tilde \phi_0^\sharp$.
\end{proof}

\section{Global Solutions: The Gubinelli-Hofmanova Approach}\label{sec:global}

The goal of this section is to prove the following theorem:
\begin{thm}
\label{thm:global-existence} Let $u$, $T^\ast$ and $\epsilon$ be as in Theorem~\ref{thm:local-solutions} and let $T < T^\ast \wedge 1$. There exists a constant $C>0$ that depends only on $\phi^\sharp_0$ and $\Z = (\ttwo{Z},Z_{0},Z_{1},Z_{2})$ but not on $T$, such that
\[
\norm{u}_{C_{T}\cC^{3/2-\epsilon}}\le C.
\]
In particular, $T^\ast = \infty$.
\end{thm}
We remind the reader that here (and in the subsequent) $Z_0,Z_1,Z_2$ are as in \eqref{eq:u-def-rewrite}.

For readability, in the remainder of this section we will suppress the dependence
of norms on time when it is clear from context, and we simply write $\norm{u}_{\alpha}=\norm{u}_{C_{T}\cC^{\alpha}}$,
with the convention that $\norm{u}_{0}=\norm{u}_{C_{T}L^{\infty}}.$

\subsection{A Paraproduct decomposition}

For $n\geq1$ to be determined later, let us consider the coupled system of equations
\begin{align*}
\left(\cL+6\nabla\ttwo{Z}\arap\nabla\right)u_{1} & =U_{1}(u_{1},u_{2})\\
\left(\cL+6\nabla\ttwo{Z}\cdot\nabla\right)u_{2} & =-e^{-6\ttwo{Z}}u_{2}^{3}+U_{2}(u_{1},u_{2}),
\end{align*}
with initial conditions $u_1(0) = u(0) = \phi^\sharp_0$ and $u_2(0) = 0$, where 
\begin{align}
U_{1}(u_{1},u_{2}) & =u^{2}\para\Delta_{>2n}Z_{2}+u\para\Delta_{>n}Z_{1}+Z_{0}\label{eq:U1-def}\\
U_{2}(u_{1},u_{2}) & =-6\nabla u_{1}\succcurlyeq\nabla\ttwo{Z}-e^{-6\ttwo{Z}}(u^{3}-u_{2}^{3})\label{eq:U2-def}\\
 & \quad+u^{2}\para\Delta_{\leq2n}Z_{2}+u\para\Delta_{\leq n}Z_{1}+u^{2}\succcurlyeq Z_{2}+u\succcurlyeq Z_{1},\nonumber 
\end{align}
for $u:=u_1+u_2$. Note here that we have implicitly suppressed the dependence of $u_{1}$
and $u_{2}$ on $n$. The idea is that $U_1$ collects all the singular contributions, while $U_2$ is more regular and in fact in $L^\infty$, so that we can apply the maximum principle to deal with it. Note that $u_1 + u_2$ solves the equation~\eqref{eq:u-def-rewrite}, so by the uniqueness proved in~\prettyref{thm:local-solutions} we have $u_1 + u_2 = u$.

We begin by estimating $U_1$ and $U_2$:
\begin{lem}
\label{lem:U1-estimates}For any $\delta\geq0$, $n\geq1$, $\epsilon'>0$, we have
\begin{align*}
\norm{U_{1}}_{-1/2-\epsilon-\delta} & \lesssim_{\Z}\left(1+2^{-n\delta}\norm{u_{1}+u_{2}}_{ 0}\right)^{2}.\\
\norm{U_{2}}_{ -1/2-\epsilon} & \lesssim_{\Z}\sum_{i=1}^{3}\norm{u_{1}}_{ 0}^{i}\cdot\norm{u_{2}}_{ 0}^{3-i}+\norm{u_{1}}_{1+\epsilon'}+\norm{u_{1}+u_{2}}_{1/2+\epsilon'}\left(1+\norm{u_{1}+u_{2}}_{ 0}\right).
\end{align*}
Furthermore, for any $\kappa > 1/2$, $n\geq1$, $\epsilon'>0$
\begin{align*}
\norm{U_{2}}_{ 0} & \lesssim_{\Z}\sum_{i=1}^{3}\norm{u_{1}}_{ 0}^{i}\cdot\norm{u_{2}}_{ 0}^{3-i}+ \norm{u_{1}}_{ 1+\epsilon'}+\norm{u_{1}+u_{2}}_{1/2+\epsilon'}\left(1+\norm{u_{1}+u_{2}}_{ 0}\right)\\
 & \qquad+ 2^{n\kappa}\norm{u_{1}+u_{2}}_{ 0}+2^{2n\kappa}\norm{u_{1}+u_{2}}_{ 0}^{2}.
\end{align*}
\end{lem}

Before turning to the proof of this result, we first recall the
following useful estimate on Littlewood-Payley blocks which follows directly from the definition: For any
$\alpha\in\R$, $f \in \cC^\alpha$ and $\delta \ge 0$,
\begin{equation}
\norm{\Delta_{>n}f}_{\cC^{\alpha-\delta}}\lesssim 2^{-n\delta}\norm{f}_{\cC^{\alpha}}\qquad\norm{\Delta_{\leq n}f}_{\cC^{\alpha+\delta}}\lesssim 2^{n\delta}\norm{f}_{\cC^{\alpha}}.\label{eq:block-bound}
\end{equation}
 We will also require the following result, which we will refer to as the paraproduct estimates.
\begin{lem}[Paraproduct Estimates, \cite{Bahouri2011}]
\label{lem:Paraproduct-Estimates} For all $\beta\in\R$, $\alpha<0$,
and $u,v\in\mathscr{S}^{'}$, the paraproduct satisfies 
\begin{align*}
\norm{u\para v}_{\cC^{\beta}} & \lesssim\norm{u}_{L^{\infty}}\norm{v}_{\cC^{\beta}}\\
\norm{u\para v}_{\cC^{\alpha+\beta}} & \lesssim\norm{u}_{\cC^{\alpha}}\norm{v}_{\cC^{\beta}}.
\end{align*}
Furthermore, if $\alpha+\beta>0$, then the resonant product satisfies
\[
\norm{u\reso v}_{\cC^{\alpha+\beta}}\lesssim\norm{u}_{\cC^{\alpha}}\norm{v}_{\cC^{\beta}},
\]
and if $\beta > 0$ also
\[
\norm{u\reso v}_{\cC^{\beta}}\lesssim\norm{u}_{L^\infty}\norm{v}_{\cC^{\beta}}.
\]
\end{lem}

With these estimates in hand, we may now prove the above lemma. 
\begin{proof}[\textbf{\emph{Proof of \prettyref{lem:U1-estimates}}}]
Let us begin with the first bound, for $U_{1} =u^{2}\para\Delta_{>2n}Z_{2}+u\para\Delta_{>n}Z_{1}+Z_{0}$.
  By \eqref{eq:U1-def} and the
paraproduct estimates (\prettyref{lem:Paraproduct-Estimates}), 
\begin{align*}
\norm{U_{1}}_{-1/2-\epsilon-\delta} & \lesssim\norm{u_{1}+u_{2}}_{ 0}^{2}\cdot \norm{\Delta_{>2n}Z_{2}}_{-1/2-\epsilon-\delta}+\norm{u_{1}+u_{2}}_{ 0}\cdot\norm{\Delta_{>n}Z_{1}}_{-1/2-\epsilon-\delta} +\norm{Z_{0}}_{-1/2-\epsilon-\delta}\\
 & \lesssim_{\Z,\epsilon}(1+2^{-n\delta}\norm{u_{1}+u_{2}}_{ 0})^{2},
\end{align*}
where in the second line we used the first estimate from \prettyref{eq:block-bound}. 

Let us now turn to the bound for $U_2$.
Since $Z_2, Z_1 \in C_T \cC^{-1/2-\epsilon'/2}$, we may apply the paraproduct estimates and the embedding $C_T \cC^{\epsilon'/2}\subset C_T L^\infty$ to find 
\begin{align*}
\norm{u^{2}\succcurlyeq Z_{2}}_{ 0}+\norm{u\succcurlyeq Z_{1}}_{ 0} & \lesssim_{\Z}\norm{u^{2}}_{1/2+\epsilon'}+\norm{u}_{1/2+\epsilon'}\lesssim \norm{u}_{1/2+\epsilon'}\left(\norm{u}_{ 0}+1\right).
\end{align*}
The gradient term is upper bounded by $\norm{\nabla u_{1}\succcurlyeq\nabla\ttwo{Z}}_0 \lesssim_\Z \norm{u_{1}}_{1+\epsilon'}$. The third order term is bounded trivially. Evidently, it remains to
bound the terms with the block sums of the type $\Delta_{\leq k}$.
To this end, we apply~\prettyref{eq:block-bound} with $\kappa' \in (1/2,(1/2+\epsilon)\wedge\kappa)$ and $\delta \ge 0$ with $\delta \neq \kappa'$:
\[
\norm{u^{2}\para\Delta_{\leq 2 n}Z_{2}}_{-\kappa'+\delta}+\norm{u\para\Delta_{\leq n}Z_{1}}_{-\kappa'+\delta}\lesssim_{\Z}2^{2n\delta}\norm{u^{2}}_{0}+2^{n\delta}\norm{u}_{0}.
\]
Choosing $\delta = 0$ and bounding $\norm{u^{2}}_0 \le \norm{u}^2_0$ yields
the first bound (since $-\kappa'>-1/2-\epsilon$). Choosing $\delta= \kappa$ yields the other (since $\kappa-\kappa'>0$). 
\end{proof}

\subsection{Controlling $u_{1}$ and $u_{2}$. }

With this in hand, we now turn to estimating the solutions $u_{1}$
and $u_{2}$ themselves. In particular, we prove the following
 corollary of the preceding lemma.
\begin{cor}
\label{cor:u_1-longtime-u-2-bound}If $n\geq0$ is such that $2^{-n(3/2-2\epsilon)}\left(\norm{u_{1}+u_{2}}_{ 0}\vee1\right)\simeq1,$
then we have for all $\alpha \in [0,3/2-\epsilon]$
\begin{align}
\norm{u_1}_{\alpha} & \lesssim_{\Z, \phi^\sharp_0} 1+ \norm{u_2}_0^{\frac{2\alpha}{3/2-\epsilon}} \label{eq:u1-bound} \\
\norm{u_{2}}_{ 0} & \lesssim_{\Z,\phi^\sharp_0}1+\norm{u_{2}}_{1/2+\epsilon'}^{1/2}\label{eq:u2-infinity-bound}\\
\norm{u_{2}}_{3/2-\epsilon} & \lesssim_{\Z, \phi^\sharp_0}1+\norm{u_{2}}_{ 0}^{3}\label{eq:u2-3/2-bound}.
\end{align}
\end{cor}

\noindent Before turning
to its proof, let us observe the following Schauder estimate.
\begin{lem}[Schauder estimates]
\label{lem:schauder} Suppose that $\alpha<1$ and $b, c \in C_{T}\cC^{-\alpha}$. For any $f\in C_{T}\cC^{2-\alpha}$
we have 
\begin{equation}
\norm{f}_{2-\alpha}\lesssim_{b,c}\norm{(\cL-b\arap\nabla-c\,\arap)f}_{-\alpha}+\norm{f(0)}_{\cC^{2-\alpha}}\label{eq:Schauder-1}.
\end{equation}
If $b \in C_T\cC^{-\alpha+1}$, then the same estimate holds for all $\alpha<2$. If $\alpha<1/2$, then we also have
\begin{equation}
	\norm{f}_{2-\alpha}\lesssim_{b,c}\norm{(\cL-b\cdot\nabla-c)f}_{-\alpha}+\norm{f(0)}_{\cC^{2-\alpha}}.\label{eq:Schauder-2}
\end{equation}

\end{lem}

\noindent It will also be helpful to note the following maximum principle. 
\begin{lem}[Maximum principle]
\label{lem:maximum-principle}Let $b\in C_{T}\cC^{-\alpha}$ for
$\alpha<1/2$ and $\Xi,g\in C_{T}C_{b}$ and $f(0,\cdot)\in C_{b}$, where $C_b$ is the space of continuous bounded functions, equipped with the supremum norm.
Then 
\begin{equation}
\norm{f}_{ 0}\lesssim_{b,\Xi}\norm{\left(\cL-b\cdot\nabla f\right)f+e^{\Xi}f^{3}}_{0}^{1/3}+\norm{f(0)}_{L^{\infty}}.\label{eq:max-prince}
\end{equation}
\end{lem}

The proof of these results follow by standard arguments and they are deferred
to the end of \prettyref{sec:global}.

We will also need the following interpolation estimate: for $f\in\mathscr{S}'$
and any $0\le \alpha\le \beta,$
\begin{equation}
\norm{f}_{\cC^{\alpha}}=\sup_{j}2^{j\alpha}\norm{\Delta_{j}f}_{L^{\infty}}\leq\sup_{j}\left(2^{j\beta}\norm{\Delta_{j}f}_{L^{\infty}}\right)^{\frac{\alpha}{\beta}}\sup_j \norm{\Delta_{j}f}_{L^{\infty}}^{1-\frac{\alpha}{\beta}} \lesssim\norm{f}_{\cC^{\beta}}^{\frac{\alpha}{\beta}}\norm{f}_{L^{\infty}}^{1-\frac{\alpha}{\beta}}.\label{eq:interpolation}
\end{equation}

\begin{proof}[\emph{\textbf{Proof of \prettyref{cor:u_1-longtime-u-2-bound}}}]
Let us begin with the estimate for $u_1$. By the Schauder estimate \eqref{eq:Schauder-1} with $b = -6 \nabla \ttwo{Z} \in \bigcap_{\epsilon'>0} C_T \cC^{-\epsilon'}$ and $c=0$
and the assumption $u(0)=\phi^\sharp_0$, we have for $\delta \in [0,3/2-\epsilon) \setminus \{1/2-\epsilon\}$:
\[
	\norm{u_{1}}_{3/2-\epsilon-\delta} \lesssim_{\Z,\phi^\sharp_0} 1 + \norm{U_{1}(u_{1},u_{2})}_{-1/2-\epsilon-\delta}.
\]
Indeed, for $\delta \in [0,1/2-\epsilon)$ the claim follows from the first version of \prettyref{eq:Schauder-1} (with $\alpha = 1/2+\epsilon+\delta < 1$), while for $\delta \in (1/2-\epsilon, 3/2-\epsilon)$ we apply the second version of \prettyref{eq:Schauder-1} (with $\alpha = 1/2+\epsilon+\delta \in  (1, 2)$).

 By the first estimate of \prettyref{lem:U1-estimates}, we have 
for $\delta \ge 0$
\[
\norm{U_{1}(u_{1},u_{2})}_{-1/2-\epsilon-\delta}\lesssim_\Z \left(1+2^{-n\delta}\norm{u_{1}+u_{2}}_{0}\right)^{2}
\]
Choosing $\delta=3/2-2\epsilon$, we see that the $L^\infty$ bound for $u_1$ (\prettyref{eq:u1-bound} with $\alpha = 0$) follows
by our choice of $n\geq1$. Now that we know that $\norm{u_1}_0 \lesssim_{\Z, \phi^\sharp_0} 1$, we apply the above estimates once more, this time with $\delta = 0$, to obtain $\norm{u_1}_{3/2-\epsilon} \lesssim_{\Z, \phi^\sharp_0} 1 + \norm{u_2}_0^2$, i.e. \prettyref{eq:u1-bound} with $\alpha = 3/2-\epsilon$. The general version of \prettyref{eq:u1-bound} then follows from the interpolation estimate \prettyref{eq:interpolation}.

Now for the estimates for $u_2$. By the maximum principle \prettyref{eq:max-prince}
and the assumption $u_{2}(0)=0$, 
\[
\norm{u_{2}}_{0}\lesssim_{\Z}\norm{U_{2}(u_{1},u_{2})}_{0}^{1/3}.
\]
By the third estimate from \prettyref{lem:U1-estimates} and the bounds
on $u_{1}$ that we just derived, we have that for any $\kappa>1/2$
and $\epsilon'>0$,
\begin{align*}
\norm{u_{2}}_{0} & \lesssim_{\Z,\phi^\sharp_0} 1+\norm{u_{2}}_{0}^{1/3}+\norm{u_{2}}_{0}^{2/3} + \norm{u_1}_{1+\epsilon'}^{1/3} +\norm{u_{1}+u_{2}}_{1/2+\epsilon'}^{1/3}\left(1+\norm{u_{1}+u_{2}}_{0}^{1/3}\right)\\
 & \qquad+2^{n\kappa/3}\norm{u_{1}+u_{2}}_{0}^{1/3}+2^{2n\kappa/3}\norm{u_{1}+u_{2}}_{0}^{2/3}\\
 & \lesssim_{\Z, \phi^\sharp_0} 1+\norm{u_{2}}_{0}^{1/3}+\norm{u_{2}}_{0}^{2/3} + \norm{u_2}_{0}^{\frac{2(1+\epsilon')}{3(3/2-\epsilon)}} + (\norm{u_{2}}_{1/2+\epsilon'} + \norm{u_2}_0^{\frac{2(1/2+\epsilon')}{3/2-\epsilon}})^{1/3}\left(1+\norm{u_{2}}_{0}^{1/3}\right)\\
 & \qquad+2^{n\kappa/3}\norm{u_{2}}_{0}^{1/3}+2^{2n\kappa/3}\norm{u_{2}}_{0}^{2/3}.
\end{align*}
By Young's inequality for products (with weights), we may subsume the powers of $\norm{u_{2}}_0$
on the left hand side.
Combining this with our assumption on $n$ yields
\begin{align*}
\norm{u_{2}}_{0} & \lesssim_{\Z,\phi^\sharp_0} 1+\norm{u_{2}}_{1/2+\epsilon'}^{1/3} \left(1+\norm{u_{2}}_{0}^{1/3}\right) +2^{2n\kappa/3}\norm{u_{2}}_{0}^{2/3}\\
& \lesssim_{\Z,\phi^\sharp_0} 1+\norm{u_{2}}_{1/2+\epsilon'}^{1/3} \left(1+\norm{u_{2}}_{0}^{1/3}\right) + \norm{u_{2}}_{0}^{\frac{2\kappa}{3(3/2-2\epsilon)}+\frac{2}{3}},
\end{align*}
for any $\kappa > 1/2$. If we take $\kappa \in (1/2, (3/2-2\epsilon)/2)$, we may again
apply Young's inequality to subsume the remaining factors of $\norm{u_{2}}_{0}$
on the left hand side to obtain
\[
\norm{u_{2}}_{0}\lesssim_{\Z}1+\norm{u_{2}}_{1/2+\epsilon'}^{1/2},
\]
which is \prettyref{eq:u2-infinity-bound}.

It remains to prove \prettyref{eq:u2-3/2-bound}. We begin as in item 1. By the
Schauder estimate~\eqref{eq:Schauder-2} for $\cL+6\nabla\ttwo{Z}\cdot\nabla$ and the assumption $u_{2}(0)=0$,
\begin{align}
\norm{u_{2}}_{3/2-\epsilon} & \lesssim_{\Z}\norm{-e^{-6\ttwo{Z}}u_{2}^{3}+U_{2}}_{-1/2-\epsilon}
  \lesssim_{\Z} \norm{u_{2}}_{0}^{3}+\norm{U_{2}}_{-1/2-\epsilon}.\label{eq:u_2-bound-lemma}
\end{align}
Recall by \prettyref{lem:U1-estimates}
and \prettyref{eq:u1-bound}, we have that 
\begin{equation}\label{eq:U_2-bound-lemma}
\begin{aligned}
\norm{U_{2}}_{-1/2+\delta}&\lesssim_{\Z, \phi^\sharp_0} 1+\norm{u_{2}}_{0}^{3} + \norm{u_2}_0^{\frac{2(1+\epsilon')}{3/2-\epsilon}} +
(\norm{u_2}_0^{\frac{2(1/2+\epsilon')}{3/2-\epsilon}}+\norm{u_{2}}_{1/2+\epsilon'})\left(1+\norm{u_{2}}_{0}\right)\\
&\lesssim_{\Z,\phi^\sharp_0} 1+\norm{u_{2}}_{0}^{3}+\norm{u_{2}}_{1/2+\epsilon'}^{3/2}.
\end{aligned}
\end{equation}
Now we apply again the interpolation estimate \prettyref{eq:interpolation} and obtain
\begin{align*}
\norm{u_{2}}_{1/2+\epsilon'}^{3/2} & \lesssim\left(\norm{u_{2}}_{3/2-\epsilon}^{\frac{1/2+\epsilon'}{3/2-\epsilon}}\norm{u_{2}}_{0}^{1-\frac{1/2+\epsilon'}{3/2-\epsilon}}\right)^{3/2} = \norm{u_{2}}_{3/2-\epsilon}^{\frac{1}{2}+\epsilon''}\cdot\norm{u_{2}}_{0}^{1-\epsilon''},
\end{align*}
for $\epsilon'' = \epsilon' + \epsilon \frac{1/2+\epsilon'}{3/2-\epsilon}$. Combining this with
\prettyref{eq:u_2-bound-lemma}-\prettyref{eq:U_2-bound-lemma}, yields
upon applying Young's inequality and re-arranging again, 
\[
\norm{u_{2}}_{3/2-\epsilon}\lesssim_{\Z,\phi^\sharp_0} 1+\norm{u_{2}}_{0}^{3}
\]
 as desired.
\end{proof}

\subsection{Proof of \prettyref{thm:global-existence}}

By \prettyref{cor:u_1-longtime-u-2-bound}, we have $\norm{u_{1}}_{3/2-\epsilon}\lesssim_{\Z,\phi^\sharp_0} 1 + \norm{u_2}_0^2$.
It remains to show that $\norm{u_2}_{3/2-\epsilon} \lesssim_{\Z,\phi^\sharp_0} 1$. By the interpolation
estimate \prettyref{eq:interpolation} 
\[
\norm{u_{2}}_{1/2+\epsilon'}\lesssim\norm{u_{2}}_{3/2-\epsilon}^{\frac13+\frac23 \epsilon''}\norm{u_{2}}_{0}^{\frac{2}{3}-\frac23\epsilon''},
\]
for $\epsilon'' = \epsilon' + \epsilon \frac{1/2+\epsilon'}{3/2-\epsilon}$ as in the proof of \prettyref{cor:u_1-longtime-u-2-bound}. Plugging this into \eqref{eq:u2-infinity-bound}, we see that for $\gamma = \frac{\epsilon''}{3}$, 
\begin{align*}
\norm{u_{2}}_{0} & \lesssim_{\Z, \phi^\sharp_0} 1+\norm{u_{2}}_{3/2-\epsilon}^{\frac16 +\gamma}\norm{u_{2}}_0^{\frac13-\gamma}\\
 & \lesssim 1+\norm{u_{2}}_{3/2-\epsilon}^{\frac14 + \frac32 \gamma}+\norm{u_{2}}_{0}^{1-3\gamma}\\
 & \lesssim_{\Z,\phi^\sharp_0} 1+\norm{u_{2}}_{0}^{\frac34 + \frac92 \gamma}+\norm{u_{2}}_{0}^{1-3\gamma},
\end{align*}
where in the last inequality, we
applied \eqref{eq:u2-3/2-bound}. As these exponents are less than
1 for $\epsilon, \epsilon'$ (and thus $\gamma$) sufficiently small, Young's inequality yields 
\[
\norm{u_{2}}_{0}\lesssim_{\Z,\phi^\sharp_0} 1.
\]
Applying this to \prettyref{eq:u1-bound} and \eqref{eq:u2-3/2-bound} then yields the desired
bound on $u_1$ and $u_2$.

To complete the proof, we still have to justify that $T^\ast = \infty$. But since we know that $\norm{u}_{C_T \cC^{3/2-\epsilon}} \le C$ for any $T< T^\ast \wedge 1$, we must have $T^\ast > 1$. Then we can repeat the same argument on the interval $[1,2]$ to see that in fact $T^\ast >2$, and so on. 
\qed

\subsection{Proof of Schauder estimates and maximum principle}

\begin{proof}[\emph{\textbf{Proof of \prettyref{lem:schauder}}}]
  We first show \prettyref{eq:Schauder-1} for $b \in C_T\cC^{-\alpha}$ and $\alpha <1$, and then indicate how to adapt the argument to obtain \prettyref{eq:Schauder-1} for $b \in C_T\cC^{-\alpha+1}$ and $\alpha<2$, as well as \prettyref{eq:Schauder-2}. By the Schauder estimates for $\cL$, we have for $\tau \in [0, T]$, for $\epsilon \in [0, 2]$, and for $\epsilon' \ge 0$ small enough so that $2-\alpha-\epsilon' > 1$ (and thus $\norm{\nabla f}_{C_\tau L^\infty} \lesssim \norm{f}_{C_\tau \cC^{2-\alpha-\epsilon'}}$):
  \begin{align} \label{eq:schauder-pr}\nonumber
    \norm{f}_{C_{\tau} \mathcal{C}^{2 - \alpha - \epsilon}} & \lesssim
    \tau^{\epsilon / 2} \norm{\mathcal{L}f}_{C_{\tau}
    \mathcal{C}^{- \alpha}} + \norm{f (0)}_{\mathcal{C}^{2 - \alpha -
    \epsilon}}\\ \nonumber
    & \lesssim  \tau^{\epsilon / 2} \norm{(\mathcal{L}- b \arap \nabla -
    c\, \arap) f}_{C_{\tau} \mathcal{C}^{- \alpha}} + \tau^{\epsilon / 2}
    \norm{ b \arap \nabla f }_{C_{\tau} \mathcal{C}^{- \alpha}}\\ \nonumber
    & \quad + \tau^{\epsilon / 2} \norm{ c\, \arap f }_{C_{\tau} \mathcal{C}^{-
    \alpha}} + \norm{ f (0) }_{\mathcal{C}^{2 - \alpha - \epsilon}}\\ \nonumber
    & \lesssim \tau^{\epsilon / 2} \norm{ (\mathcal{L}- b \arap \nabla - c\,
    \arap) f }_{C_{\tau} \mathcal{C}^{- \alpha}} + \tau^{\epsilon / 2} (\norm{
    b }_{C_T \mathcal{C}^{- \alpha}} + \norm{ c }_{C_T \mathcal{C}^{-
    \alpha}}) \norm{ f }_{C_{\tau} \cC^{2-\alpha-\epsilon'}}\\ 
    &\quad + \norm{ f (0) }_{\mathcal{C}^{2 - \alpha - \epsilon}}.
  \end{align}
  We choose $\epsilon = \epsilon' > 0$ and $\tau > 0$ small enough so that with the implicit constant $C$ on the right hand side
  $C \tau^{\epsilon / 2} (\norm{ b }_{C_T \mathcal{C}^{- \alpha}} + \norm{ c
  }_{C_T \mathcal{C}^{- \alpha}}) \leqslant \frac{1}{2}$. In that way we obtain
  \[ \norm{ f }_{C_{\tau} \mathcal{C}^{2 - \alpha - \epsilon}} \lesssim
     \tau^{\epsilon / 2} \norm{ (\mathcal{L}- b \arap \nabla - c \,\arap) f
     }_{C_{\tau} \mathcal{C}^{- \alpha}} + \norm{ f (0) }_{\mathcal{C}^{2 -
     \alpha - \epsilon}} . \]
  As $\tau$ can be chosen independently of $f (0)$ and as $\norm{ f (\tau)
  }_{\mathcal{C}^{2 - \alpha - \epsilon}} \lesssim \norm{ f }_{C_{\tau}
  \mathcal{C}^{2 - \alpha - \epsilon}}$, we can now iterate this on $[\tau,
  2 \tau]$, $[2 \tau, 3 \tau]$, \dots, and obtain
  \[ \norm{ f }_{C_{\tau} \mathcal{C}^{2 - \alpha - \epsilon}} \lesssim \norm{
     (\mathcal{L}- b \arap \nabla - c \arap) f }_{C_T \mathcal{C}^{- \alpha}}
     + \norm{ f (0) }_{\mathcal{C}^{2 - \alpha}} . \]
  To obtain the bound for $\epsilon=0$ we simply apply~\eqref{eq:schauder-pr} once more, this time for $\tau = T$ and
  $\epsilon = 0$ and with the same $\epsilon'>0$ as before.
  
  If $b \in C_T \cC^{-\alpha+1}$ for $\alpha < 2$, then we only use the $C_\tau L^\infty$ regularity of $f$ and obtain $\norm{b\arap\nabla f}_{C_\tau \cC^{-\alpha}} \lesssim \norm{b}_{C_\tau \cC^{-\alpha+1}} \norm{f}_{C_\tau L^\infty}$. Since $2-\alpha > 0$ and thus $C_\tau \cC^{2-\alpha-\epsilon} \subset C_\tau L^\infty$ for all sufficiently small $\epsilon >0$, we can then use the same argument as before.
  
  Nearly the same arguments also work for \prettyref{eq:Schauder-2}, except that now we have to make sure that the products $b\cdot \nabla f$ and $c \nabla f$ are well-defined. This is the case if $-\alpha + 2-\alpha - 1 > 0$, i.e., if $\alpha < 1/2$.
\end{proof}
\begin{proof}[\emph{\textbf{Proof of \prettyref{lem:maximum-principle}}}]
 We assume first that $b,\Xi,g,f(0,\cdot)$ are smooth functions,
so that $f\in C^{1,2}$ is a classical solution, and we write

\[
(\partial_{t}-\Delta-b\cdot\nabla)f=-e^{\Xi}f^{3}+g.
\]
By compactness, $f$ takes its minimum and maximum on $[0,T]\times \bT^3$. Let $f(t,x)$ be a local maximum. If $t=0$, then $|f(t,x)|\leqslant\|f(0,\cdot)\|_{L^{\infty}}$.
If $t \in (0,T)$, then $\partial_{t}f(t,x)=\partial_{i}f(t,x)=0$, while
$\Delta f(t,x)\leqslant0$. Therefore

\[
-e^{\Xi(t,x)}f^{3}(t,x)+g(t,x)=(\partial_{t}-\Delta-b\cdot\nabla)f(t,x)\geqslant0,
\]
and thus

\[
f(t,x)\leqslant\left(\frac{g(t,x)}{e^{\Xi(t,x)}}\right)^{1/3}\leqslant\frac{\|g\|_{C_{T}L^{\infty}}^{1/3}}{e^{-\|\Xi\|_{C_{T}L^{\infty}}/3}}.
\]
If $f$ attains its maximum in $(T,x)$, then we must have 
\[
\partial_{t}f(T,x)\geqslant0
\]
 because otherwise $f(t,x)>f(T,x)$ for some $t<T$. Furthermore, $x$ is 
a local maximum of $f(T,\cdot)$ so that we still have $\partial_{i}f(T,x)=0$
and $\Delta f(T,x)\leqslant0$. Thus, we have again $(\partial_{t}-\Delta-b\cdot\nabla)f(t,x)\geqslant0$
and we get the same estimate as before.

To control the minimum, it suffices to note that minima of $f$ are
maxima of $-f$ and that $(\partial_{t}-\Delta-b\cdot\nabla)(-f)=-e^{\Xi}(-f)^{3}-g$.

Our estimate does not depend on $b$ at all, and it only depends on
$\Xi$ through its $L^{\infty}$ norm. For $b\in C_{T}\mathcal{C}^{-\alpha}$
with $\alpha<1/2$ it follows from our Schauder estimates that if
we approximate $b_{n}\rightarrow b$ in $C_{T}\mathcal{C}^{-\alpha-}$
and $\Xi_{n},g_{n}\rightarrow\Xi$ in $C_{T}L^{\infty}$ and $f_{n}(0,\cdot)\rightarrow f(0,\cdot)$
in $L^{\infty}$ with $b_{n},\Xi_{n},g_{n},f_{n}(0,\cdot)$ smooth,
then the solutions $f_{n}$ to

\[
(\partial_{t}-\Delta-b_{n}\cdot\nabla)f_{n}=-e^{\Xi_{n}}f^{3}+g_{n},\qquad f_{n}(0)=f_{n}(0,\cdot),
\]
converge locally uniformly to $f$, and thus the estimate also holds
for $f$.
\end{proof}

\section{Possible extensions}\label{sec:extensions}

In recent years there have been many works dedicated to the $\Phi^4_3$ equation, and by now we know much more than global existence and uniqueness of solutions on the torus. We expect that with the multiplicative transform some proofs of more refined properties can be simplified. To keep the paper short and accessible we do not attempt this here, but let us sketch some possible extensions of our results:

\begin{itemize}
	\item \textbf{Initial condition:} We could easily extend the local solution theory to $\phi_0^\sharp \in \cC^{-\alpha}$ by working in spaces $C([0,T], \cC^{-\alpha}) \cap \mathcal M^\gamma_T \cC^{1+\epsilon'}$ with $\norm{u}_{\mathcal M^{\gamma}_T \cC^{1+\epsilon'}} = \sup_{t \in [0,T]} t^\gamma \norm{u(t)}_{\cC^{1+\epsilon'}}$, for $\gamma = \frac{1+\epsilon'+\alpha}{2}$, see \cite[Section 6]{Gubinelli2017KPZ} for details. To control the nonlinearity $u^3$ while keeping the singularity integrable in time we need $\alpha<2/3$. Since $Z(0) \in \cC^{-1/2-\epsilon}$, we can thus take any element of $\cC^{-2/3+\epsilon}$ as initial condition and the special form $\phi_0 = Z(0)+\phi^\sharp_0$ is not necessary.

	 \item \textbf{More singular noise:} For applications in quantum field theory we are only interested in the invariant measure on $\bT^3$ (or on $\R^3$ if we consider the equation on the whole space). Therefore, the physically meaningful dimension would be $1+3=4$, i.e. the $\Phi^4_4$ equation on $\R_+ \times \bT^4$ resp. $\R_+ \times \R^4$. (Here the lower index in $\Phi^k_d$ indicates the space dimension in the equation, the upper index $k$ means that the leading order nonlinearity is $\phi^{k-1}$. Recall that in the stochastic quantization approach \cite{Parisi1981}, the ``stochastic time", $\R_+$, is an extra dimension and does not  correspond to``real'' time.) But in $d=4$ the space-time white noise is more singular and the $\Phi^4_4$ equation is \emph{critical} in the sense of Hairer \cite{Hairer2014}, which roughly means that the solution $u$ to our transformed equation is expected to have the same regularity as $\phi$. There are currently no mathematical tools for the $\Phi^4_4$ equation. But it is possible to treat a proxy for ``$4-\epsilon$ dimensions'', by considering a slightly regularized noise in $d=4$ \cite{Bruned2017Renormalising, Chandra2019}.
	 
	 	Our approach works if the regularity of $-6\nabla \ttwo{Z}$ is bigger than $-1/2$, and also bigger than $-1$ minus the minimum regularity of $Z_0, Z_1, Z_2$. For more singular noise it might be possible to further improve the regularity of $Z_0, Z_1$ by additional additive and multiplicative transformations. But $Z_2$ will never be more regular than $Z$. In ``$d=4-2\epsilon$'' we have $Z \in C_T \cC^{-1+\epsilon}$ and $\nabla \ttwo{Z} \in C_T \cC^{-1+2\epsilon}$. Therefore, we need $\epsilon > 1/3$. Beyond that it might be possible to additionally use a Zvonkin transform as in \cite{Zhang2020} to remove the term $-6\nabla \ttwo{Z} \cdot \nabla u$. But this becomes quite complicated, and for very singular noise it seems easier to use the arguments of \cite{Chandra2019}.
	 
	 \item \textbf{Coming down from infinity:} By adapting Lemma~2.7 in \cite{Moinat2020} we could derive an alternative version of the maximum principle:
	 	\begin{equation}
			\norm{f(t)}_{L^\infty}\lesssim_{b,\Xi} \max\left\{t^{-1/2}, \norm{\left(\cL-b\cdot\nabla f\right)f+e^{\Xi}f^{3}}_{0}^{1/3}\right\},
		\end{equation}
		where the right hand side is independent of $f(0)$. With this it should be possible to control $u(1)$ independently of the initial condition. As in \cite{Mourrat2017Dynamic} this should lead to a proof of existence of stationary measures for $\phi$ via the Krylov-Bogoliubov method.
		
	\item \textbf{Existence of moments:} The constant $C$ in Theorem~\ref{thm:global-existence} depends on $e^{3\ttwo{Z}}$ and on $e^{-6\ttwo{Z}}$. Since $\ttwo{Z}$ is a second order polynomial of a Gaussian process, it is therefore not even clear if $\E[\norm{u}_{C_{1}\cC^{3/2-\epsilon}}] < \infty$, nor if after transforming back we could get $\E[\norm{\phi}_{C_1\cC^{-1/2-\epsilon}}]<\infty$. While it is actually known that $\phi$ has stretched exponential moments \cite{Moinat2020} and, at least at stationarity, even $\E[\exp(\lambda \langle \phi(t), \eta \rangle^4)]<\infty$ for suitable test functions $\eta$ and sufficiently small $\lambda>0$ \cite{Hairer2021}. To try recovering such results with our approach, we would have to slightly change the multiplicative transformation and multiply with $e^{3 \Delta_{>m} \ttwo{Z}}$ instead, where $m$ is large enough so that $\norm{\Delta_{>m} \ttwo{Z}}_{C_T \cC^\alpha} \simeq 1$ for a suitable $\alpha < 1-\epsilon$.
		
	 \item \textbf{Extension to the whole space:} To solve the equation on $\R_+ \times \R^3$ we would have to use weighted spaces. Then we have to be careful when implementing the exponential transform, because $e^{3 \ttwo{Z}(x)}$ grows too fast as $|x| \to \infty$. We should therefore apply \cite[Lemma 2.4]{Gubinelli2019Global} to decompose $\ttwo{Z} = \ttwo{Z}_> + \ttwo{Z}_<$, where $\ttwo{Z}_> \in C_T \cC^{1-\epsilon}$ (without weight) and $\ttwo{Z}_<$ is in a weighted space but has regularity better than $1$. And then we use the multiplicative transform with $e^{3 \ttwo{Z}_>}$. After that we expect that the construction of solutions on $\R_+ \times \R^3$ can be handled similarly as for the $\Phi^4_2$ equation on $\R_+ \times \R^2$ in \cite{Gubinelli2019Global}.
\end{itemize}

\appendix

\section{Some nonlinear functions of the noise}

Let $\mathcal{T} = (Z, \llbracket Z^2 \rrbracket, \tthree{Z}, \ttwo{Z}, \tthree{Z}\reso Z, \ttwo{Z}\reso \llbracket Z^2 \rrbracket - \frac b3, |\nabla \ttwo{Z}|^2 - \frac b3,\tthree{Z} \reso \llbracket Z^2 \rrbracket - b Z)$, with the regularities given in \prettyref{sec:local}.
In the preceding, we used that the trees $\mathcal{T}$ are well-defined after renormalization and that $Y$ and a certain product involving its derivative
are continuous functions of these trees and, in particular, are well-behaved under renormalization. For the convenience of the reader, we collect these results here.

\begin{lem}
	If $\xi$ is a white noise on $\R \times \bT^3$, then all terms in $\mathcal{T}$ are well-defined as (sums of) iterated stochastic integrals, and they have the regularities given in Table~\ref{t:reg}.
\end{lem}

\begin{proof}
	This is shown in \cite{Catellier2018, Mourrat2017Construction}, except that there the term $|\nabla \ttwo{Z}|^2 - \frac b3$ does not appear. Note that, since the paraproduct is always well-defined, the existence of this term is equivalent to the existence of the renormalized resonant product $\sum_i \partial_i \ttwo{Z}\reso \partial_i \ttwo{Z} - \frac b3$. As a space-time distribution, this latter term is equivalent to $\ttwo Z \reso \llbracket Z^2 \rrbracket - \frac b3$: Indeed, by applying Leibniz's rule to $\cL (\ttwo{Z} \reso \ttwo{Z})$ we obtain
	\begin{align*}
		\sum_i \partial_i \ttwo{Z}\reso \partial_i \ttwo{Z} -\frac b3& =  \ttwo{Z} \reso \cL \ttwo{Z} -\frac b3 - \cL (\ttwo{Z} \reso \ttwo{Z}) - \ttwo{Z} \reso \ttwo{Z} \\
		& =  \ttwo{Z} \reso \llbracket Z^2 \rrbracket -\frac b3 - \cL (\ttwo{Z} \reso \ttwo{Z}) - \ttwo{Z} \reso \ttwo{Z}.
	\end{align*}
	Since $\ttwo{Z}$ has positive regularity, all the terms on the right hand side are well-defined. However, it is not clear if $\cL (\ttwo{Z} \reso \ttwo{Z})$ is in a $C_T \cC^\alpha$ space and we can only treat it as a space-time distribution.
	
	Therefore, we need to use stochastic arguments to construct $\sum_i \partial_i \ttwo{Z}\reso \partial_i \ttwo{Z} - \frac b3$. Using again Leibniz's rule, this time for $\Delta$, we see that constructing $\sum_i \partial_i \ttwo{Z}\reso \partial_i \ttwo{Z} - \frac b3 \in C_T \cC^{-\epsilon}$ is equivalent to constructing $\ttwo{Z}\reso \Delta \ttwo{Z} - \frac b3$. But this is nearly the same as $\ttwo{Z} \reso \llbracket Z^2 \rrbracket - \frac b3$, except that we replace $\llbracket Z^2 \rrbracket(t)$ by $\int_0^t \Delta P_{t-s} \llbracket Z^2 \rrbracket(s) \mathrm d s$. The kernel $\int_0^t \Delta P_{t-s}\mathrm d s$ neither gains nor loses regularity, and therefore (nearly) the same argument as in \cite[Section 4.5]{Catellier2018} or \cite[Section 4.2]{Mourrat2017Construction} show that $\ttwo{Z}\reso \Delta \ttwo{Z} - \frac b3 \in C_T \cC^{-\epsilon}$.
\end{proof}

To prove the desired continuity estimates for $Y$ and a certain product involving its derivative, we need some tools
from paracontrolled distributions. First we recall the paralinearization theorem (see \cite{Bahouri2011} or \cite[Lemma 2.6]{Gubinelli2015Paracontrolled}) which says that for $0<\alpha<1$, $F\in C^{2}$ and $f\in \cC^\alpha$ we have 
\begin{equation}\label{eq:paralinearization}
F(f)-F'(f)\para f \in \cC^{2\alpha}.
\end{equation}
Next, we recall Bony's paramultiplication bound \cite[Theorem 2.3]{Bony1981} which says that 
if $f,g\in \cC^\alpha$ with non-integer $\alpha>0$ and $h\in\cC^\beta$, then 
\begin{equation}\label{eq:paramultiplication}
\norm{f \para (g\para h) - (fg)\para h}_{\alpha+\beta} \lesssim \norm{f}_\alpha\norm{g}_\alpha \norm{h}_\beta.
\end{equation}
Next, let us recall the well-known fact that if we let $Jf(t) = \int_0^t P_{t-s} f ds $, then $J$ approximately commutes
with paraproducts in the sense that for all $\alpha \in (0,1)$, $\beta \in \R$ and $\delta < 2$
\begin{equation}\label{eq:j-commutation}
		\norm{J(f\para g) - f\para J g}_{C_T \cC^{\alpha+\beta+\delta}} \lesssim (\norm{f}_{C_T \cC^\alpha} + \norm{f}_{C^{\frac\alpha 2}_T L^\infty}) \norm{g}_{C_T \cC^\beta},
\end{equation}
where $\norm{f}_{C^{\frac\alpha 2}_T L^\infty} = \sup_{0\le s < t \le T} \frac{\norm{f(t) - f(s)}_{L^\infty}}{|t-s|^{\frac \alpha 2}}$.
(While to the best of our knowledge this exact bound does not appear in the literature, it is a folklore result; see, e.g.,  \cite[Lemma 5.5.7]{Perkowski2014} or \cite[Lemma A.1]{Kremp2020} for very similar estimates from which this follows.)
Finally, we recall the commutator estimate \cite[Lemma 2.4]{Gubinelli2015Paracontrolled}  which says that
\begin{equation}\label{eq:commutator}
C(f,g,h) = (f\para g)\reso h - f(g\reso h)
\end{equation}
is a bounded trilinear map from $\cC^\alpha \times \cC^\beta \times \cC^\gamma$ to $\cC^{\alpha+\beta+\gamma}$, provided that $\alpha \in (0,1)$ and $\beta + \gamma < 0$ while $\alpha+\beta+\gamma>0$. We now show that $Y$ and a certain product involving its derivative are continuous in $\mathcal T$. (And consequently, they are well-defined in the renormalization limit.)

\begin{lem}
	Let
	\[
		\cL Y = 3e^{3\ttwo{Z}}( \tthree{Z} \llbracket Z^2 \rrbracket - b(Z + \tthree{Z})),\quad Y(0) = 0.
	\]
	Then $Y \in \bigcap_{\epsilon>0} C_T\cC^{1-\epsilon}$ is a continuous function of $\mathcal{T}$. Moreover,
	\[
		\nabla \ttwo{Z} \cdot \nabla Y - b e^{3\ttwo{Z}} \tthree{Z} \in \bigcap_{\epsilon>0} C_T \cC^{-2\epsilon}
	\]
	is also a continuous function of $\mathcal T$.
\end{lem}

\begin{proof}
	We split up
	\[
		\tthree{Z} \llbracket Z^2 \rrbracket - b(Z + \tthree{Z}) = \left( \tthree{Z} \reso \llbracket Z^2 \rrbracket - b Z\right) + \left(\tthree{Z} \para \llbracket Z^2 \rrbracket - b \tthree{Z}\right) + \tthree{Z} \arap \llbracket Z^2 \rrbracket.
	\]
	By assumption, $\left( \tthree{Z} \reso \llbracket Z^2 \rrbracket - b Z\right) \in C_T \cC^{-\frac12 - \epsilon}$, and since $\ttwo{Z} \in C_T \cC^{1-\epsilon}$ and the regularity is preserved under the smooth function $\exp(\cdot)$, the product $3e^{3\ttwo{Z}}\left( \tthree{Z} \reso\llbracket Z^2 \rrbracket - bZ \right)$ is well-defined and in $C_T \cC^{-\frac12 -\epsilon}$. Similarly, $\tthree{Z} \arap \llbracket Z^2 \rrbracket \in C_T \cC^{-\frac12 - \epsilon}$ by the paraproduct estimates, and therefore also $3e^{3\ttwo{Z}}\tthree{Z} \arap \llbracket Z^2 \rrbracket \in C_T \cC^{-\frac12 - \epsilon}$. The most complicated term is the one involving $\tthree{Z} \para \llbracket Z^2 \rrbracket$. This paraproduct is of course well-defined, but its regularity is only $C_T \cC^{-1-\epsilon}$ and therefore the resonant product $3e^{3\ttwo{Z}} \reso \left(\tthree{Z} \para \llbracket Z^2 \rrbracket\right)$ is not defined.
	
	Thus we need to use the more refined tools from paracalculus described above to construct this product as a continuous function of $\mathcal T$. First note that by the paralinearization theorem \eqref{eq:paralinearization} we have $e^{3\ttwo{Z}} - e^{3\ttwo{Z}} \para 3\ttwo{Z} \in C_T \cC^{2-2\epsilon}$, and therefore this difference can be multiplied with $\tthree{Z} \para \llbracket Z^2 \rrbracket$. Then, by the commutator estimate \eqref{eq:commutator}, we have that
	\begin{align*}
		3e^{3\ttwo{Z}} \reso ( \tthree{Z} &\para \llbracket Z^2 \rrbracket)-3e^{3\ttwo{Z}}b\tthree{Z} \\ & = (3e^{3\ttwo{Z}} - 9 e^{3\ttwo{Z}}\para \ttwo{Z}) \reso (\tthree{Z} \para \llbracket Z^2 \rrbracket) + 9 C(e^{3\ttwo{Z}}, \ttwo{Z}, \tthree{Z} \para \llbracket Z^2 \rrbracket) \\
		&\qquad \qquad+ 9 e^{3\ttwo{Z}}( \ttwo{Z} \reso (\tthree{Z} \para \llbracket Z^2 \rrbracket)) - 3e^{3\ttwo{Z}} b \tthree{Z} \\
		& = C_T \cC^{1-3\epsilon} + 9 e^{3\ttwo{Z}} \left(C(\tthree{Z}, \llbracket Z^2 \rrbracket, \ttwo{Z}) + \tthree{Z}\left((\llbracket Z^2\rrbracket\reso \ttwo{Z}) - \frac{b}{3}\right)\right),
	\end{align*}
	where on the right hand side we used the notation $C_T \cC^{1-3\epsilon}$ to denote a term of this regularity which is given as a continuous function of the data. The commutator on the right hand side has positive regularity, so it no longer poses a problem. The term $(\llbracket Z^2\rrbracket\reso \ttwo{Z}) - \frac{b}{3}$ is in $C_T \cC^{-\epsilon}$ by assumption, so we can multiply it with $\tthree{Z} \in C_T \cC^{\frac12-\epsilon}$ and the product is in $C_T \cC^{-\epsilon}$. Then we can multiply this product with $e^{3\ttwo{Z}} \in C_T \cC^{1-\epsilon}$, and therefore the remaining expression on the right hand side is well-defined and in $C_T \cC^{-\epsilon} \subset C_T \cC^{-\frac12-\epsilon}$.
	
	Next, we consider the renormalized product $\nabla \ttwo{Z} \cdot \nabla Y - b e^{3\ttwo{Z}} \tthree{Z}$. The idea is to write $Y$ as a paraproduct plus a more regular remainder, and to apply the commutator estimate for $C$ as above. By the previous considerations and the paraproduct estimates, $\cL Y - 3e^{3\ttwo{Z}}\para(\tthree{Z} \para \llbracket Z^2 \rrbracket) \in C_T \cC^{-\frac12-\epsilon}$. By Bony's paramultiplication bound \eqref{eq:paramultiplication}, we moreover have
	\[
		\norm{3e^{3\ttwo{Z}}\para(\tthree{Z} \para \llbracket Z^2 \rrbracket) - (3e^{3\ttwo{Z}}\tthree{Z}) \para \llbracket Z^2 \rrbracket}_{C_T \cC^{-3\epsilon}} \lesssim \norm{3e^{3\ttwo{Z}}}_{C_T \cC^{\frac12 - \epsilon}} \norm{\tthree{Z}}_{C_T \cC^{\frac12 - \epsilon}} \norm{\llbracket Z^2 \rrbracket}_{C_T \cC^{-1-\epsilon}},
	\]
	so that $\cL Y - (3e^{3\ttwo{Z}}\tthree{Z}) \para \llbracket Z^2 \rrbracket \in C_T \cC^{-\frac12-\epsilon}$. By the Schauder estimates for $\cL$ we deduce that $Y - J((3e^{3\ttwo{Z}}\tthree{Z}) \para \llbracket Z^2 \rrbracket) \in C_T \cC^{\frac32-\epsilon}$. 
		Therefore, the product $\nabla \ttwo{Z} \cdot \nabla (Y - J((3e^{3\ttwo{Z}}\tthree{Z}) \para \llbracket Z^2 \rrbracket)) \in C_T \cC^{-\epsilon}$ is well-defined. To deal with the term $\nabla \ttwo{Z} \cdot \nabla J((3e^{3\ttwo{Z}}\tthree{Z}) \para \llbracket Z^2 \rrbracket)$, we need to use the commutator estimate \eqref{eq:j-commutation}. To apply this,  however, we first need time regularity of $3e^{3\ttwo{Z}}\tthree{Z}$. For $\tthree{Z}$ it follows from the estimate (1.11) in \cite{Mourrat2017Construction} that $\norm{\tthree{Z}}_{C^{\frac14 - \frac \epsilon 2}_T L^\infty} < \infty$. For $\ttwo{Z}$ it follows from the Schauder estimates in \cite[Lemma 2.10]{Gubinelli2019Global} that $\norm{\ttwo{Z}}_{C^{\frac 12 - \frac\epsilon 2}_T L^\infty} \lesssim \norm{\llbracket Z^2 \rrbracket}_{C_T \cC^{-1-\epsilon}} < \infty$. Therefore, using that $J\llbracket Z^2 \rrbracket = \ttwo{Z}$, we can apply \eqref{eq:j-commutation} to get 
	\[
		J((3e^{3\ttwo{Z}}\tthree{Z}) \para \llbracket Z^2 \rrbracket) - (3e^{3\ttwo{Z}}\tthree{Z}) \para \ttwo{Z} \in C_T \cC^{\frac32-3\epsilon},
	\]
	and it remains to make sense of $\nabla [(3e^{3\ttwo{Z}}\tthree{Z}) \para \ttwo{Z}] \cdot \nabla \ttwo{Z}$. By Leibniz's rule and the paraproduct estimates we have $\nabla [(3e^{3\ttwo{Z}}\tthree{Z}) \para \ttwo{Z}] -  (3e^{3\ttwo{Z}}\tthree{Z}) \para \nabla \ttwo{Z} \in C_T \cC^{\frac12 -2\epsilon}$. Then we apply the commutator estimate for $C$ to reduce the last ill-defined product to
	\begin{align*}
		& \sum_i ((3e^{3\ttwo{Z}}\tthree{Z}) \para \partial_i \ttwo{Z})\reso \partial_i \ttwo{Z} - b e^{3\ttwo{Z}} \tthree{Z}\\
		&\qquad = \sum_i C(3e^{3\ttwo{Z}}\tthree{Z}, \partial_i \ttwo{Z}, \partial_i \ttwo{Z}) + 3e^{3\ttwo{Z}}\tthree{Z} (\sum_i (\partial_i \ttwo{Z} \reso \partial_i \ttwo{Z}) - \frac b 3).
	\end{align*}
	Now it suffices to note that $\sum_i (\partial_i \ttwo{Z} \reso \partial_i \ttwo{Z}) - \frac b 3 = |\nabla \ttwo{Z}|^2 - \frac b3 - 2\sum_i (\partial_i \ttwo{Z} \para \partial_i \ttwo{Z})$ and that the paraproducts are well-defined and in $C_T \cC^{-2\epsilon}$.
\end{proof}

\section*{Acknowledgements} 
The first author gratefully acknowledges support from NSERC, cette recherche a \'et\'e financ\'ee par CRSNG [RGPIN-2020-04597, DGECR-2020-00199]. The second author gratefully acknowledges financial support by the DFG via Research Unit FOR2402.

\bibliographystyle{amsalpha} 
\bibliography{all} 

\end{document}